\documentclass[12pt, reqno]{amsart}
\usepackage[utf8]{inputenc}
\usepackage{amsmath}
\usepackage{amsfonts}
\usepackage{amssymb}
\usepackage{amsthm}
\usepackage{bm}
\usepackage{bbm}
\usepackage{xcolor}
\usepackage{url}
\usepackage[hidelinks]{hyperref}
\usepackage{mathtools}

\usepackage{graphicx,tikz}

\numberwithin{equation}{section}

\theoremstyle{plain}
\newtheorem{theorem}{Theorem}[section]
\newtheorem{lemma}[theorem]{Lemma}

\newtheorem{proposition}[theorem]{Proposition}

\newtheorem{remark}[theorem]{Remark}

\title[Improved exponents for Erd\H{o}s--Selfridge curves]{Improved exponent bounds for Erd\H{o}s--Selfridge curves}
\author{Pranabesh Das}
\address{Xavier University of Louisiana, Department of Mathematics, New Orleans, LA 70125, USA}
\email{pranabesh.math@gmail.com}
\author{Kyle Pratt}
\address{Brigham Young University, Department of Mathematics, Provo, UT 84602, USA}
\email{kyle.pratt@mathematics.byu.edu}


\allowdisplaybreaks

\begin{document}
\date{}

\begin{abstract}
Given integers $k,\ell \geq 2$, define the Erd\H{o}s--Selfridge curve
\begin{align*}
\mathcal{C}_{k,\ell} : y^\ell = x(x+1)\cdots (x+k-1).
\end{align*}
Bennett and Siksek proved that if $(x,y)$ is a rational point on $\mathcal{C}_{k,\ell}$ with $\ell$ a prime and $y \neq 0$, then $\ell$ is bounded (doubly-exponentially) in terms of $k$. We improve this bound on $\ell$ when $k$ is sufficiently large. Our method uses combinatorial and sieve-theoretic arguments, as well as results on solutions to generalized Fermat equations.
\end{abstract}

\maketitle

\section{Introduction}

\subsection{History and results.}

Let $k$ and $\ell$ be integers $\geq 2$, and consider the Erd\H{o}s--Selfridge curve
\begin{align}\label{eq:main ES curve}
\mathcal{C}_{k,\ell} : y^\ell = x(x+1)\cdots (x+k-1).
\end{align}
These curves are named for Erd\H{o}s and Selfridge, who proved \cite{ES1975} that there are no solutions in positive integers $n$ and $y$ to the equation
\begin{align*}
y^\ell = n(n+1) \cdots (n+k-1).
\end{align*}
Another, perhaps more modern, way to view the Erd\H{o}s--Selfridge result is to see it as a statement about the integer points on the curves $\mathcal{C}_{k,\ell}$. In particular, it follows from the work of Erd\H{o}s and Selfridge that
\begin{align*}
\mathcal{C}_{k,\ell}(\mathbb{Z}) = \{(-n,0) : n \in \mathbb{Z}, 0 \leq n \leq k-1\}.
\end{align*}

In addition to determining the integer points on the Erd\H{o}s--Selfridge curves, it is natural to ask for a determination of $\mathcal{C}_{k,\ell}(\mathbb{Q})$. If $k + \ell \geq 7$, then the genus of (the smooth projective model of) $\mathcal{C}_{k,\ell}$ is at least two, and it follows from work of Faltings \cite{Fal1983} that $\mathcal{C}_{k,\ell}(\mathbb{Q})$ is finite. Hence, it is possible, at least in principle, to list the rational points on the Erd\H{o}s--Selfridge curves.

Clearly \eqref{eq:main ES curve} has rational points with $y=0$; we may term these the ``trivial'' rational points. Of greater interest is to determine whether $\mathcal{C}_{k,\ell}$ has ``nontrivial'' rational points with $y \neq 0$. A conjecture of Sander \cite{San1999} (corrected in \cite{BS2016} and \cite{Pratt2024}) posits that these nontrivial rational points can only appear in certain cases when $\ell=2$, or when $(k,\ell)=(3,3)$. For example, we have $(-3/2,3/4) \in \mathcal{C}_{4,2}(\mathbb{Q})$, and $(-4/3,2/3) \in \mathcal{C}_{3,3}(\mathbb{Q})$.

Rigorously establishing Sander's conjecture seems quite difficult. The current state-of-the-art is that Sander's conjecture is solved in the following cases: if $k$ is small and fixed (see, e.g., \cite{BBGH2006,GHP2009,San1999}); if $k$ is large, $\ell$ is small compared to $k$, and $\gcd(k,\ell)=1$, due to the second author \cite{Pratt2024}; if $\ell$ is large in terms of $k$, due to Bennett and Siksek \cite{BS2016}. Our interest in this paper is in this last case, where $\ell$ is large in terms of $k$. According to Sander's conjecture, we then expect $\mathcal{C}_{k,\ell}$ to have only trivial rational points. We remark that the method of Bennett and Siksek has since been extended to other related situations (see \cite{DLS2018, DLSS2023, Edis2019}).

We need to introduce some notation. Given a positive real number $z$, we write
\begin{align}\label{eq:primorial}
z\# \coloneqq \prod_{p \leq z} p
\end{align}
for the primorial of $z$, the product of all primes $\leq z$. The prime number theorem implies $z\# = \exp((1+o(1))z)$ as $z \rightarrow \infty$.

Bennett and Siksek \cite{BS2016} obtained the following result.

\begin{theorem}\label{thm:Bennett Siksek theorem}
Let $k \geq 2$ and $\ell \geq 2$ be prime. If $(x,y) \in (\mathbb{Q}^\times)^2$ is a nontrivial rational point on $\mathcal{C}_{k,\ell}$, then
\begin{align*}
\log \ell \leq (1+o(1)) \frac{8}{3}\cdot \frac{\log k}{k} \cdot k\#,
\end{align*}
where $o(1)$ denotes a quantity that goes to zero as $k$ goes to infinity.
\end{theorem}

\begin{remark}
Bennett and Siksek do not state their result as in Theorem \ref{thm:Bennett Siksek theorem}, but rather give the simpler bound $\log\ell < 3^k$. See Remark \ref{rmk:sharper bennett siksek theorem} below for further discussion of this point.
\end{remark}

Theorem \ref{thm:Bennett Siksek theorem} bounds $\ell$ by a doubly-exponential function of $k$, and it is obviously of interest to reduce this bound as much as possible. In this paper, we make some modest progress towards this goal. Our main result is as follows.

\begin{theorem}\label{thm:main theorem}
Let $k$ be sufficiently large, and let $\ell$ be a prime. If $(x,y) \in (\mathbb{Q}^\times)^2$ is a nontrivial rational point on $\mathcal{C}_{k,\ell}$, then
\begin{align*}
\log \ell \ll k^6 \log k \cdot \frac{k\#}{B\#},
\end{align*}
where $B = k^{1/(100\log \log k)}$.
\end{theorem}

Theorem \ref{thm:main theorem} improves on Theorem \ref{thm:Bennett Siksek theorem} by a factor of $\approx e^B$, which one might view as a ``quasi-exponential'' or ``almost-exponential'' improvement.

\begin{remark}
Using a more involved sieve construction (see Section \ref{sec:sieve methods}), our method would allow one to take $B$ in Theorem \ref{thm:main theorem} as large as $k^{\alpha}$, for some constant $\alpha > 0$. It seems challenging to prove a version of Theorem \ref{thm:main theorem} with $B \gg k$.
\end{remark}

Edis \cite{Edis2019} obtained a version of Theorem \ref{thm:Bennett Siksek theorem} when some terms $x+i$ are dropped from \eqref{eq:main ES curve}. It is likely that our method would similarly allow one to obtain an analogue of Theorem \ref{thm:main theorem} with some terms $x+i$ dropped.

\subsection{Notation.} We define the primorial $z\#$ as in \eqref{eq:primorial}. For a positive integer $n$, we define
\begin{align*}
    \text{Rad}(n) = \prod_{p \mid n} p, \ \ \ \ \text{Rad}_2(n) = \prod_{\substack{p \mid n \\ p > 2}} p.
\end{align*}
We write $\omega(n)$ for the number of distinct prime factors of $n$. We denote the prime-counting function by $\pi(x)$. Throughout the paper, $p$ and $q$ always denote prime numbers.

\subsection{Outline of the proof and structure of the paper.} In Section \ref{sec:reduction to key propositions}, we reduce the problem of determining rational points on $\mathcal{C}_{k,\ell}$ to a Diophantine equation in integer variables. In particular, we study the Diophantine equation
\begin{align*}
    t^\ell = n(n+d^\ell)\cdots (n+(k-1)d^\ell),
\end{align*}
where $n,t,d$ are integers with certain properties that we suppress here (see Lemma \ref{lem:reduction to integrality}). We state two key propositions, Propositions \ref{prop:key prop d divisible by a small prime} and \ref{prop:key prop d not divisible by any small primes}, and then give the proof of Theorem \ref{thm:main theorem} assuming the propositions. Proposition \ref{prop:key prop d divisible by a small prime} applies when $d$ is divisible by a small prime $p$ (where ``small'' means $p < k$), and Proposition \ref{prop:key prop d not divisible by any small primes} applies when $d$ is not divisible by any small primes.

In Section \ref{sec:preliminaries}, we set out some basic tools for the rest of the argument. The terms $n+id^\ell$ have factorizations $n+id^\ell = a_ix_i^\ell$ (see Lemma \ref{lem:factorization lemma}), and much of our method revolves around studying facets of the integers $a_i$. Some weak control on the size of the $a_i$ is provided by an argument of Erd\H{o}s (see Lemma \ref{lem:good set I to control ai}). We close the section with a statement of the modularity of elliptic curves over $\mathbb{Q}$ (Lemma \ref{lem:basic modularity}) and an application of modularity to bounding the exponents of solutions in generalized Fermat equations (Lemma \ref{lem:modularity to bound exponent ell}).

We prove the easier of the two key propositions, Proposition \ref{prop:key prop d divisible by a small prime}, in Section \ref{sec:proof of easier key prop}, using the tools laid out in the previous section.

We begin the proof of Proposition \ref{prop:key prop d not divisible by any small primes} in Section \ref{sec:reduction of key prop to another prop}. By assumption, we have that $d$ is not divisible by any prime $p < k$. It is technically convenient to work with the primes in $(k/2,k)$. Given such a prime $p$, since $p \nmid d$ we see that $p \mid (n+id^\ell)$ if and only if $i \equiv i_p \pmod{p}$ for some integer $i_p$. It follows that each prime $p \in (k/2,k)$ divides at least one term $n+id^\ell$, and at most two such terms. We must break into cases based on this distinction, and we break into further cases depending on whether there is some term $n+i_0d^\ell$ that is divisible by ``many'' primes $<k$. If there is some $n+i_0d^\ell$ divisible by many primes, then Proposition \ref{prop:some term divisible by many primes < k} bounds $\ell$. The remainder of the section is devoted to combinatorial arguments that reduce the proof of Proposition \ref{prop:key prop d not divisible by any small primes} to Proposition \ref{prop:some term divisible by many primes < k}.

We prove Proposition \ref{prop:some term divisible by many primes < k} in the last section of the paper, Section \ref{sec:sieve methods}. The proof is fundamentally based on sieve methods. (We use only the basic Brun sieve, but one can use more sophisticated sieves to obtain stronger results.) If every term $n+id^\ell$ is not divisible by ``too many'' primes $<k$, then the method of Bennett--Siksek already gives an improvement. Thus, we can assume there is some $n+i_0d^\ell$ that is divisible by ``almost all'' primes $<k$. In this situation, we can mostly determine which primes divide other terms $n+jd^\ell$. Using a sieve and results on the distribution of primes in arithmetic progressions, we can find ``many'' primes $p$ such that $n+(i_0 \pm p)d^\ell$ and $n+(i_0\pm cp)d^\ell$ have a very limited set of prime divisors, where $c$ is a small integer like $2$ or $3$. After accounting for some minor technicalities, we can then form a generalized Fermat equation whose solutions can be controlled by work of Kraus \cite{Kra1997} (Lemma \ref{lem:Kraus fermat lemma}).

\section{Reduction to key propositions}\label{sec:reduction to key propositions}

The first task is reducing the problem of rational points on $\mathcal{C}_{k,\ell}$ to a Diophantine equation involving integer variables.

\begin{lemma}\label{lem:reduction to integrality}
Let $k,\ell \geq 2$ be coprime. Let $(x,y)$ be a rational point on $\mathcal{C}_{k,\ell}$ with $y \neq 0$. Then there are nonzero integers $n,t,d$ with $\gcd(n,d)=1, d \geq 1$ such that
\begin{align*}
t^\ell = n(n+d^\ell)\cdots (n+(k-1)d^\ell).
\end{align*}
\end{lemma}
\begin{proof}
We may write $x = \frac{n}{b}, y = \frac{t}{c}$, where $n$ and $t$ are nonzero integers, $b$ and $c$ are positive integers, and $\gcd(n,b)=\gcd(t,c)=1$. Inserting the expressions for $x$ and $y$ into \eqref{eq:main ES curve} and multiplying through, we have
\begin{align}\label{eq:reduction to integrality intermediate}
t^\ell b^k = c^\ell n(n+b)\cdots(n+(k-1)b).
\end{align}
Since $\gcd(n,b)=\gcd(t,c)=1$, we have $b^k \mid c^\ell$ and $c^\ell \mid b^k$. As $b$ and $c$ are both positive, this implies $b^k = c^\ell$. By our hypothesis that $k$ and $\ell$ are coprime, there exists a positive integer $d$ such that $b=d^\ell,c=d^k$. We conclude by cancelling $b^k=c^\ell$ from both sides of \eqref{eq:reduction to integrality intermediate}.
\end{proof}

Observe that if $p$ is a prime such that $p$ divides $n(n+d^\ell)\cdots (n+(k-1)d^\ell$ as in Lemma \ref{lem:reduction to integrality}, then the $p$-adic valuation of $n(n+d^\ell)\cdots (n+(k-1)d^\ell$ is divisible by $\ell$.

We can now state our key propositions. The first proposition is useful when $d$ is divisible by a somewhat small prime.

\begin{proposition}\label{prop:key prop d divisible by a small prime}
Let $k$ be sufficiently large, and let $\ell > k$ be a prime. Suppose there are nonzero integers $t,n,d$ with $\gcd(n,d)=1$ and $d\geq 1$ such that
\begin{align*}
t^\ell = n(n+d^\ell)\cdots (n+(k-1)d^\ell).
\end{align*}
If $d$ is divisible by an odd prime $p_0 < k$, then
\begin{align*}
\log \ell \ll k^2 \log p_0.
\end{align*}
\end{proposition}

The next proposition applies when $d$ is not divisible by any primes $<k$.

\begin{proposition}\label{prop:key prop d not divisible by any small primes}
Let $k$ be sufficiently large, and let $\ell > k$ be a prime. Suppose there are nonzero integers $t,n,d$ with $\gcd(n,d)=1$ and $d\geq 1$ such that
\begin{align*}
t^\ell = n(n+d^\ell)\cdots (n+(k-1)d^\ell).
\end{align*}
Assume $d$ is not divisible by any odd prime $< k$. Then
\begin{align*}
\log \ell \ll k^6 \log k \cdot \frac{k\#}{B\#},
\end{align*}
where $B= k^{1/(100\log \log k)}$.
\end{proposition}

\begin{proof}[Proof of Theorem \ref{thm:main theorem} assuming Propositions \ref{prop:key prop d divisible by a small prime} and \ref{prop:key prop d not divisible by any small primes}]
We assume $k$ is sufficiently large, and that $\ell$ is a prime. If $\ell \leq k$ then the theorem holds, so we may assume $\ell > k$. It follows that $k$ and $\ell$ are coprime.

Let $(x,y) \in (\mathbb{Q}^\times)^2$ be a nontrivial rational point on $\mathcal{C}_{k,\ell}$. By Lemma \ref{lem:reduction to integrality}, there are nonzero integers $t,n,d$ with $\gcd(n,d)=1$ and $d \geq 1$ such that
\begin{align*}
t^\ell = n(n+d^\ell)\cdots (n+(k-1)d^\ell).
\end{align*}
If $d$ is divisible by some odd prime $<k$, then the theorem follows from Proposition \ref{prop:key prop d divisible by a small prime}. On the other hand, if $d$ is not divisible by any odd prime $<k$, then the theorem follows from Proposition \ref{prop:key prop d not divisible by any small primes}.
\end{proof}

\section{Preliminaries}\label{sec:preliminaries}

\begin{lemma}\label{lem:factorization lemma}
Let $\ell \geq 3$ be an odd prime. Suppose we have integers $n,t,d$ as in the conclusion of Lemma \ref{lem:reduction to integrality}. Then for $0 \leq i \leq k-1$ we have $n+id^\ell = a_i x_i^\ell$, where $a_i$ is a positive integer that is $\ell$-power-free and divisible only by primes $\leq k-1$. Furthermore, for $0 \leq i < j \leq k-1$, we have $\gcd(a_i,a_j) \mid (j-i)$.
\end{lemma}
\begin{proof}
Since $n$ and $d$ are coprime, we have $\gcd(n+id^\ell,n+jd^\ell)\mid (j-i)$ for $0 \leq i < j \leq k-1$. Hence, if there is a prime $p \geq k$ that divides $t$, then $p$ divides a single term $n+id^\ell$. We may then write $n+id=b_iz_i^\ell$, where $b_i$ is only divisible by primes $\leq k-1$, and $z_i$ is only divisible by primes $\geq k$. We obtain the conclusion of the lemma by pulling $\ell$-th powers and minus signs on $b_i$ into $z_i^\ell$.
\end{proof}

The factorizations in Lemma \ref{lem:factorization lemma} are fundamental, and we use Lemma \ref{lem:factorization lemma} throughout the paper without explicit reference.

\begin{lemma}\label{lem:good set I to control ai}
Let the integers $a_i$ be as in the conclusion of Lemma \ref{lem:factorization lemma}. There exists a set $I \subseteq \{0,1,\ldots,k-1\}$ with $|I| \geq k-\pi(k)$ such that
\begin{align*}
\prod_{i \in I} a_i \leq k!.
\end{align*}
\end{lemma}
\begin{proof}
This is essentially \cite[Lemma 2]{ES1975}.
\end{proof}

Whenever we write $I$ in this paper, we always mean the set $I$ as in the conclusion of Lemma \ref{lem:good set I to control ai}.

We need some results on solutions to generalized Fermat equations. As usual, modularity and level-lowering plays a key role.

\begin{lemma}\label{lem:basic modularity}
Let $\ell$ be a rational prime. Let $E/\mathbb{Q}$ be an elliptic curve of conductor $N$ and $f(z)=z+\sum_{i\geq 2} c_i z^i$ be a newform of weight $2$ and level $N' \mid N$. Write $K = \mathbb{Q}(c_1,c_2,\ldots)$ for the totally real number field generated by the Fourier coefficients of $f$. If $\overline{\rho}_{E,\ell}$ arises from $f$, then there is some prime $\lambda \mid \ell$ of $K$ such that for all primes $q$:
\begin{itemize}
\item if $q \nmid \ell N N'$, then $a_q(E) \equiv c_q \pmod{\lambda}$;
\item if $q \nmid \ell N'$ and $q \| N$, then $q+1 \equiv \pm c_q \pmod{\lambda}$.
\end{itemize}
\end{lemma}
\begin{proof}
This is \cite[Lemma 3.1]{BS2016}.
\end{proof}

We can use modularity to bound the exponents in generalized Fermat equations. Note that, for a positive integer $n$, we have
\begin{align*}
    \text{Rad}_2(n) = \prod_{\substack{p \mid n \\ p > 2}} p.
\end{align*}

\begin{lemma}\label{lem:modularity to bound exponent ell}
Let $\ell >163$ be a prime. Let $a,b,c,u,v,w$ be nonzero integers with
\begin{align*}
au^\ell+bv^\ell+cw^\ell = 0,
\end{align*}
and let $p\neq \ell$ be an odd prime, such that:
\begin{enumerate}
\item $a,b,c$ are $\ell$-power-free,
\item $p\nmid abc$,
\item $p$ divides exactly one of $u$, $v$, or $w$.
\end{enumerate}
Then
\begin{align*}
\log \ell \leq \frac{32\textup{Rad}_2(abc)+1}{6} \log(\sqrt{p}+1).
\end{align*}
\end{lemma}
\begin{proof}
This follows from arguments of Bennett and Siksek \cite[Lemma 2.1 and Section 3]{BS2016}, but we give the details for completeness. (See also \cite[Lemma 6]{Edis2019}.)

By removing any potential common factors, we may assume $au^\ell,bv^\ell,cw^\ell$ are pairwise coprime while still maintaining the hypotheses of the lemma. (The analysis requires a bit of casework, and is easier if one remembers that whenever we have nonzero integers $x+y+z=0$ with $v_p(x) \geq v_p(y) \geq v_p(z)$, say, then we must have $v_p(y)=v_p(z)$.) We may permute and change signs as necessary to assume
\begin{align*}
au^\ell \equiv -1 \pmod{4}, \ \ \ \ bv^\ell\equiv 0 \pmod{2}.
\end{align*}
Having done so, we form the Frey curve
\begin{align*}
E : Y^2 = X(X-au^\ell)(X+bv^\ell).
\end{align*}
With $E$ in hand, we can then consider the Galois representation $\overline{\rho}_{E,\ell}: \text{Gal}(\overline{\mathbb{Q}}/\mathbb{Q}) \rightarrow \text{GL}_2(\mathbb{F}_\ell)$. Since $\ell >163$, work of Mazur \cite[Theorem 1]{Maz1978} implies the representation $\overline{\rho}_{E,\ell}$ is irreducible.

By work of Kraus \cite[Sections 3 and 4]{Kra1997}, the conductor of $E$ is equal to $N = 2^r \text{Rad}_2(abcuvw)$, where $0 \leq r \leq 5$, and $E$ is related via level-lowering to a newform $f$ of weight $2$ and level $N' \mid 2^5 \text{Rad}_2(abc)$. Write $K$ for the totally real number field generated by $f$.

Our hypotheses imply $p \mid N$ and $p \nmid N'$. By Lemma \ref{lem:basic modularity}, we have
\begin{align*}
\ell \mid N_{K/\mathbb{Q}}(p+1\pm c_p).
\end{align*}
Every real embedding of $c_p$ is bounded in absolute value by $2\sqrt{p}$, so this norm is nonzero and we have
\begin{align*}
\ell \leq (\sqrt{p}+1)^{2[K:\mathbb{Q}]}.
\end{align*}
The degree $[K:\mathbb{Q}]$ of the number field is bounded above by the dimension of the space of cuspidal newforms of weight 2 and level $N'$. This dimension is $\leq \frac{N'+1}{12}$ by work of Martin \cite[Theorem 2]{Mar2005}, so
\begin{align*}
\log \ell &\leq \frac{N'+1}{6}\log(\sqrt{p}+1) \leq \frac{32\text{Rad}_2(abc)+1}{6} \log(\sqrt{p}+1).\qedhere
\end{align*}
\end{proof}

\section{Proof of Proposition \ref{prop:key prop d divisible by a small prime}}\label{sec:proof of easier key prop}

The proof of Proposition \ref{prop:key prop d divisible by a small prime} follows, in a mostly straightforward way, from the results in the previous section.

\begin{proof}[Proof of Proposition \ref{prop:key prop d divisible by a small prime}]
Let $p_0$ be an odd prime that divides $d$. For any $0 \leq j < k-1$, we then have
\begin{align*}
d^\ell = (n+(j+1)d^\ell) - (n+jd^\ell) = a_{j+1}x_{j+1}^\ell - a_j x_j^\ell,
\end{align*}
so
\begin{align*}
d^\ell + a_jx_j^\ell + a_{j+1}(-x_{j+1})^\ell = 0.
\end{align*}
Since $p_0 \mid d$ and $\gcd(n,d)=1$, we have $p_0 \nmid a_ja_{j+1}x_jx_{j+1}$. It follows by Lemma \ref{lem:modularity to bound exponent ell} that
\begin{align*}
\log \ell &\leq \frac{32\textup{Rad}_2(a_ja_{j+1})+1}{6} \log(\sqrt{p_0}+1).
\end{align*}
Since $\text{Rad}_2(a_ja_{j+1}) \leq a_ja_{j+1}$, in order to complete the proof it suffices to show we can find some $j$ such that $a_j a_{j+1} \ll k^2$.

Recall the set $I$ in the conclusion of Lemma \ref{lem:good set I to control ai}. For $0 \leq j \leq k-2$, let $P_j = \{j,j+1\}$. Observe that $P_0,P_2,P_4,\ldots$ are disjoint, and there are $\frac{k}{2} + O(1)$ such sets $P_j$ that are subsets of $\{0,1,\ldots,k-1\}$. Let us say $P_j$, with $j$ even, is ``good'' if $j,j+1 \in I$, and $P_j$ is ``bad'' if $j \not \in I$ or $j+1 \not \in I$. Since the sets $P_j$ are disjoint, the number of bad $P_j$ is $\leq \pi(k)$ because each bad $P_j$ must contain at least one element of $\{0,1,\ldots,k-1\}\backslash I$. Therefore, the number $N$ of good $P_j$ is $\geq \frac{k}{2} - O(\pi(k))$. 

It follows from Lemma \ref{lem:good set I to control ai} that
\begin{align*}
\prod_{P_j \text{ good}}a_ja_{j+1} \leq k!.
\end{align*}
Taking the minimal product, we then see there is some $j$ such that
\begin{align*}
a_ja_{j+1} \leq (k!)^{\frac{1}{N}} \leq (k!)^{\frac{2}{k}(1+O(\frac{1}{\log k}))}.
\end{align*}
Stirling's formula then implies $a_ja_{j+1} \ll k^2$.
\end{proof}

\section{Reduction of Proposition \ref{prop:key prop d not divisible by any small primes} to Proposition \ref{prop:some term divisible by many primes < k}}\label{sec:reduction of key prop to another prop}

The proof of Proposition \ref{prop:key prop d not divisible by any small primes} is substantially more involved than the proof of Proposition \ref{prop:key prop d divisible by a small prime}. Similar to the method of Bennett and Siksek, we use primes $p \in (\frac{k}{2},k)$, but our analysis is much more intricate. 

By assumption, $p \nmid d$ for all odd primes $p < k$. Therefore, for each prime odd $p < k$, there exists some $i_p$ such that $p$ divides $n+id^\ell$ if and only if $i \equiv i_p \pmod{p}$; we shall use this fact repeatedly in what follows without further comment. In particular, each prime $p \in (\frac{k}{2},k)$ divides at least one term $n+i_pd^\ell$, and divides at most two terms $n+i_pd^\ell,n+(i_p+p)d^\ell$. We must break into cases depending on which of these situations occur, and we must also break into cases depending on how many primes $<k$ divide terms $n+id^\ell$. The following result is our main tool when there is some term $n+id^\ell$ divisible by many primes $<k$.

\begin{proposition}\label{prop:some term divisible by many primes < k}
Let $k$ be sufficiently large, and let $\ell > k$ be a prime. Suppose there are nonzero integers $t,n,d$ with $\gcd(n,d)=1$ and $d\geq 1$ such that
\begin{align*}
t^\ell = n(n+d^\ell)\cdots (n+(k-1)d^\ell).
\end{align*}
Assume $d$ is not divisible by any odd prime $< k$, and assume there is some $0 \leq i_0 \leq k-1$ such that $n+i_0d^\ell$ is divisible by $>\pi(k)-100 k^{1/(100\log \log k)}$ primes $<k$. Then
\begin{align*}
\log \ell \ll k \log k.
\end{align*}
\end{proposition}

Proposition \ref{prop:some term divisible by many primes < k} immediately applies in some cases, but, in other cases, it takes some nontrivial effort to apply it. The following subsection is devoted to proving some of the results we need in one of these knotty situations.

\subsection{Every prime divides two terms}

In this subsection, we assume that every prime $p \in (\frac{k}{2},k)$ divides two terms $n+i_pd^\ell,n+(i_p+p)d^\ell$. For a prime $p \in (\frac{k}{2},k)$, we write
\begin{align*}
S_p \coloneqq \{q < k : q \text{ prime}, q \mid (n+i_pd^\ell)(n+(i_p+p)d^\ell)\}.
\end{align*}
We assume further that $|S_p| \geq \frac{2}{3}\pi(k)$ for every $p \in (\frac{k}{2},k)$.

We may write $S_p$ as the disjoint union $S_p = A_p \cup B_p \cup C_p$, where
\begin{align*}
A_p &\coloneqq \{q < k : q \text{ prime}, q \mid (n+i_pd^\ell), q \nmid(n+(i_p+p)d^\ell)\}, \\
B_p &\coloneqq \{q < k : q \text{ prime}, q \mid (n+i_pd^\ell), q \mid(n+(i_p+p)d^\ell)\}, \\
C_p &\coloneqq \{q < k : q \text{ prime}, q \nmid (n+i_pd^\ell), q \mid(n+(i_p+p)d^\ell)\}.
\end{align*}
Observe that $p \in B_p$, by assumption. On the other hand, if $q \in B_p$, then $q$ divides the difference $(n+(i_p+p)d^\ell)-(n+i_pd^\ell) = pd^\ell$, and since $\gcd(n,d)=1$ we must have $q \mid p$, so $q=p$. It follows that $B_p = \{p\}$, and therefore
\begin{align}\label{eq:lower bound on size of Ap plus Cp}
\frac{2}{3}\pi(k)-1 \leq |S_p|-1 \leq |A_p\cup C_p| = |A_p| + |C_p|.
\end{align}

\begin{lemma}\label{lem:Sp large implies indices collide}
Let $p,r$ be two distinct primes in $(\frac{k}{2},k)$, and assume $|S_p|,|S_r| \geq \frac{2}{3}\pi(k)$. Then $i_p = i_r$, or $i_p+p=i_r+r$.
\end{lemma}
\begin{proof}
We observe that we cannot have $i_p+p=i_r$ or $i_p=i_r+r$. Indeed, suppose $i_p+p=i_r$. Then
\begin{align*}
k > i_r+r > i_r + \frac{k}{2} = i_p+p + \frac{k}{2} > i_p+k \geq k,
\end{align*}
which is a contradiction. We similarly cannot have $i_p=i_r+r$. Hence, we have $i_p +p \neq i_r$, and $i_p \neq i_r+r$.

Now assume by way of contradiction that $i_p \neq i_r$ and $i_p+p \neq i_r+r$. By inclusion-exclusion, we have
\begin{align*}
|(A_p\cup C_p)\cap (A_r \cup C_r)| &= |A_p\cup C_p| + |A_r \cup C_r| - |(A_p\cup C_p)\cup (A_r \cup C_r)| \\
&\geq 2\left(\frac{2}{3}\pi(k)-1\right) - \pi(k) = \frac{1}{3}\pi(k)-2,
\end{align*}
where we have used \eqref{eq:lower bound on size of Ap plus Cp} to obtain a lower bound.

On the other hand, we have
\begin{align*}
(A_p\cup C_p)\cap (A_r \cup C_r) &= (A_p \cap A_r) \cup (A_p \cap C_r) \cup (C_p \cap A_r) \cup (C_p \cap C_r).
\end{align*}
Observe that
\begin{align*}
|A_p \cap A_r| &\leq \#\{q < k : q \mid (n+i_pd^\ell),q \mid (n+i_rd^\ell)\} \leq \#\{q < k : q \mid (i_p-i_r)\}.
\end{align*}
Since $i_p \neq i_r$, we see that $i_p-i_r$ is a nonzero integer with absolute value $<k$. Accordingly, the number of primes $<k$ that divide $i_p-i_r$ is $\ll \log k$. Similarly, we have
\begin{align*}
|A_p \cap C_r|, |C_p \cap A_r|,  |C_p \cap C_r| \ll \log k.
\end{align*}
Comparing our two expressions, we find
\begin{align*}
\pi(k) \ll \log k,
\end{align*}
and this is a contradiction since $k$ is sufficiently large.
\end{proof}

\begin{lemma}\label{lem:indices all align}
Assume $|S_p| \geq \frac{2}{3}\pi(k)$ for every $p \in (\frac{k}{2},k)$. Then $i_p=i_r$ for all primes $p,r \in (\frac{k}{2},k)$, or $i_p+p=i_r+r$ for all primes $p,r \in (\frac{k}{2},k)$.
\end{lemma}
\begin{proof}
Let $p_1 < p_2$ be the two smallest primes in $(\frac{k}{2},k)$. By Lemma \ref{lem:Sp large implies indices collide}, we have $i_{p_1}=i_{p_2}$ or $i_{p_1}+p_1 = i_{p_2} + p_2$.

Assume that $i_{p_1} = i_{p_2}$. Let $q \in (\frac{k}{2},k)$ be any prime with $q > p_2$. We claim that $i_q = i_{p_1}$, so assume for contradiction that $i_q \neq i_{p_1}$. By Lemma \ref{lem:Sp large implies indices collide}, we must then have $i_q + q = i_{p_1}+p_1$. However, we have $i_q \neq i_{p_2} = i_{p_1}$, and we must also have $i_q +q \neq i_{p_2} + p_2$, for if $i_q + q = i_{p_2}+p_2$ then $i_{p_1}+p_1 = i_q + q = i_{p_2}+p_2 = i_{p_1} + p_2$, and this implies $p_1 = p_2$. This contradicts Lemma \ref{lem:Sp large implies indices collide}, so we must have $i_q = i_{p_1}$.

We argue similarly in the case $i_{p_1} + p_1 = i_{p_2} + p_2$.
\end{proof}

Lemma \ref{lem:indices all align} says that the indices $i_p$ must all be equal, or the indices $i_p+p$ must all be equal. With the indices colliding in this fashion, we can find some index $i \in \{i_p,i_p+p\}$ where $n+id^\ell$ is not divisible by many primes $<k$. To see this, we define a set of indices $\mathcal{R}$ by
\begin{align*}
\mathcal{R} = \left\{i_p+p : p \in \left( \frac{k}{2},k\right) \right\}
\end{align*}
if $i_p=i_r$ for all $p,r \in (\frac{k}{2},k)$, and 
\begin{align*}
\mathcal{R} = \left\{i_p : p \in \left( \frac{k}{2},k\right) \right\}
\end{align*}
if $i_p+p=i_r+r$ for all $p,r \in (\frac{k}{2},k)$. Observe that all the elements of $\mathcal{R}$ are distinct by Lemma \ref{lem:indices all align}, and $|\mathcal{R}| \sim \frac{k}{2\log k}$ by the prime number theorem.

\begin{lemma}\label{lem:some j in R not divisible by many primes}
If $k$ is sufficiently large, there is some $j \in \mathcal{R}$ such that
\begin{align*}
\#\{p < k : p\textup{ prime}, p \mid(n+jd^\ell)\} \leq 3(\log k)(\log \log k).
\end{align*}
\end{lemma}
\begin{proof}
By double-counting, we have
\begin{align*}
\sum_{j \in \mathcal{R}} &\#\{p < k :p\text{ prime}, p \mid(n+jd^\ell)\}=\sum_{j \in \mathcal{R}}\sum_{\substack{p < k \\ j \equiv i_p (p)}} 1 = \sum_{p < k} \sum_{\substack{j \in \mathcal{R} \\ j \equiv i_p(p)}} 1 \\
&\leq \sum_{p < k} \sum_{\substack{0 \leq j < k \\ j \equiv i_p (p)}} 1 = \sum_{p < k} \left(\frac{k}{p} + O(1) \right) = k \log \log k + O \left(k\right),
\end{align*}
the last equality following from Mertens' theorem. Therefore, there exists some $j \in \mathcal{R}$ such that
\begin{align*}
\#\{p < k :p\text{ prime}, p \mid(n+jd^\ell)\} &\leq (1+o(1))\frac{k \log \log k}{|\mathcal{R}|} \leq 3(\log k)(\log \log k)
\end{align*}
when $k$ is sufficiently large, since $|\mathcal{R}| = (1+o(1))\frac{k}{2\log k}$.
\end{proof}

By Lemma \ref{lem:some j in R not divisible by many primes}, there exists some $j_0 \in \mathcal{R}$ such that 
\begin{align*}
\#\{p < k : p\textup{ prime}, p \mid(n+j_0d^\ell)\} \leq 3(\log k)(\log \log k).
\end{align*}
Since $j_0 = i_{p_0}$ or $j_0 = i_{p_0}+p_0$ for some $p_0 \in (\frac{k}{2},k)$, we have $|A_{p_0}| \leq 3(\log k)(\log \log k)$ or $|C_{p_0}| \leq 3(\log k)(\log \log k)$. By \eqref{eq:lower bound on size of Ap plus Cp}, we then see that
\begin{align*}
\max(|A_{p_0}|,|C_{p_0}|) \geq |S_{p_0}| - 3(\log k)(\log \log k)-1.
\end{align*}

Thus, we have proved the following lemma.

\begin{lemma}\label{lem:every prime divides two, some single term has many prime divisors}
Assume $|S_p| \geq \frac{2}{3}\pi(k)$ for every prime $p \in (\frac{k}{2},k)$. There exists some $p \in (\frac{k}{2},k)$ such that
\begin{align*}
\max(|A_p|,|C_p|) \geq |S_p| - 3(\log k)(\log \log k)-1.
\end{align*}
\end{lemma}

\subsection{Proof of Proposition \ref{prop:key prop d not divisible by any small primes}}

With Lemma \ref{lem:every prime divides two, some single term has many prime divisors} in hand, we are ready to prove Proposition \ref{prop:key prop d not divisible by any small primes}, assuming Proposition \ref{prop:some term divisible by many primes < k}.

\begin{proof}[Proof of Proposition \ref{prop:key prop d not divisible by any small primes} assuming Proposition \ref{prop:some term divisible by many primes < k}]
Assume $p \nmid d$ for every odd prime $p < k$. We set 
\begin{align*}
B = k^{1/(100 \log \log k)},
\end{align*}
and now consider several cases.

Assume there exists a prime $p \in (\frac{k}{2},k)$ such that $p$ divides only a single term $n+i_pd^\ell$. Note that this implies $p \mid x_{i_p}$. 

If $n+i_pd^\ell$ is divisible by $>\pi(k)-B$ of the primes $<k$, then we may apply Proposition \ref{prop:some term divisible by many primes < k} to bound $\ell$. 

Therefore, we may assume $n+i_pd^\ell$ is divisible by $\leq \pi(k)-B$ of the primes $<k$. Thus, $a_{i_p}$ is also divisible by $\leq \pi(k)-B$ of the primes $<k$. For any $0 \leq j \leq k-1$ with $j \neq i_p$, we subtract $n+i_pd^\ell = a_{i_p}x_{i_p}^\ell$ and $n+jd^\ell = a_jx_j^\ell$ and rearrange to obtain
\begin{align*}
a_{i_p}x_{i_p}^\ell + a_j(-x_j)^\ell + (j-i_p)d^\ell = 0.
\end{align*}

Observe that $p \nmid x_jd$, and $p \nmid a_{i_p}a_j(j-i_p)$. Therefore, by Lemma \ref{lem:modularity to bound exponent ell}, we have
\begin{align*}
\log \ell \leq \frac{32\textup{Rad}_2(a_{i_p}a_j|j-i_p|)+1}{6} \log(\sqrt{k}+1) .
\end{align*}
We use the trivial inequalities $\text{Rad}_2(abc) \leq \text{Rad}_2(a)\cdot \text{Rad}_2(b) \cdot \text{Rad}_2(c)$ and $\text{Rad}_2(n) \leq n$ to get
\begin{align*}
\textup{Rad}_2(a_{i_p}a_j|j-i_p|) \leq ka_j \text{Rad}_2(a_{i_p}).
\end{align*}

By Lemma \ref{lem:good set I to control ai}, there exists some $j \in I\backslash \{i_p\}$ such that
\begin{align*}
a_j \leq (k!)^{1/|I\backslash \{i_p\}|} \leq (k!)^{\frac{1}{k}(1+O(\frac{1}{\log k}))} \ll k,
\end{align*}
the last inequality following from Stirling's formula. For this choice of $j$, we have
\begin{align*}
\textup{Rad}_2(a_{i_p}a_j|j-i_p|) \ll k^2 \text{Rad}_2(a_{i_p}).
\end{align*}
Since $a_{i_p}$ is divisible by $\leq \pi(k)-B$ primes $<k$, we have the (somewhat crude) bound
\begin{align*}
\text{Rad}_2(a_{i_p}) &\leq \frac{\prod_{q \leq k} q}{\prod_{q \leq B} q} = \frac{k\#}{B\#},
\end{align*}
which establishes the proposition in this case (and, in fact, the result is slightly stronger).

On the other hand, assume that every prime $p \in (\frac{k}{2},k)$ divides two terms $n+i_p d^\ell,n+(i_p+p)d^\ell$. If every such $p$ has
\begin{align*}
\#\{q < k : q \text{ prime}, q \mid (n+i_pd^\ell)(n+(i_p+p)d^\ell)\} &> \pi(k)-B,
\end{align*}
then since $B = o(\pi(k))$ Lemma \ref{lem:every prime divides two, some single term has many prime divisors} implies there is some $0 \leq i_0 \leq k-1$ such that $n+i_0d^\ell$ is divisible by $>\pi(k)-B-3(\log k)(\log \log k)-1$ primes $<k$. We can then appeal to Proposition \ref{prop:some term divisible by many primes < k} to conclude. 

Therefore, we may assume there is some $p \in (\frac{k}{2},k)$ such that
\begin{align*}
\#\{q < k : q \text{ prime}, q \mid (n+i_pd^\ell)(n+(i_p+p)d^\ell)\} &\leq \pi(k)-B.
\end{align*}
For any positive integers $r,s$ with $r+s=p$, we have
\begin{align*}
(n+i_pd^\ell)(n+(i_p+p)d^\ell) - (n+(i_p+r)d^\ell)(n+(i_p+s)d^\ell) + rs d^{2\ell} = 0,
\end{align*}
and this yields the generalized Fermat equation
\begin{align*}
a_{i_p}a_{i_p+p}(x_{i_p}x_{i_p+p})^\ell + a_{i_p+r}a_{i_p+s}(-x_{i_p+r}x_{i_p+s})^\ell + rs (d^2)^\ell = 0.
\end{align*}
We observe $p \mid x_{i_p}x_{i_p+p}$, and $p \nmid a_{i_p}a_{i_p+p}a_{i_p+r}a_{i_p+s}rsx_{i_p+r}x_{i_p+s}d$. Thus, by Lemma \ref{lem:modularity to bound exponent ell}, we obtain
\begin{align*}
\log \ell &\leq \frac{32\textup{Rad}_2(a_{i_p}a_{i_p+p}a_{i_p+r}a_{i_p+s}rs)+1}{6} \log(\sqrt{k}+1).
\end{align*}

We trivially have $rs \leq k^2$, and since $a_{i_p}a_{i_p+p}$ is divisible by $\leq \pi(k)-B$ of the primes $<k$, we have
\begin{align*}
\text{Rad}_2(a_{i_p}a_{i_p+p}) \leq \frac{k\#}{B\#}.
\end{align*}

It remains to show there is a choice of $r$ and $s$ such that $a_{i_p+r}a_{i_p+s} \ll k^4$. The argument we employ for this purpose is similar to that used in the proof of Proposition \ref{prop:key prop d divisible by a small prime}. 

Since $r+s=p$, we have $s=p-r$, so $i_p+s=i_p+p-r$. If $r < \frac{p}{2}$, then $i_p+r < i_p+p-r$. We accordingly define sets $Q_r = \{i_p+r,i_p+p-r\}$ for $1 \leq r < \lfloor \frac{p}{2}\rfloor$, and observe that these $\frac{p}{2} + O(1)$ sets, each with two elements, are pairwise disjoint. There are $\leq \pi(k)$ choices of $r$ where $i_p +r \not \in I$ or $i_p+p-r \not \in I$. Therefore, by Lemma \ref{lem:good set I to control ai}, there exists some $r$ such that
\begin{align*}
a_{i_p+r}a_{i_p+p-r} &\leq (k!)^{\frac{2}{p}(1+O(\frac{1}{\log k}))} \leq (k!)^{\frac{4}{k}(1+O(\frac{1}{\log k}))} \ll k^4. \qedhere
\end{align*}
\end{proof}

\begin{remark}\label{rmk:sharper bennett siksek theorem}
The proof of \cite[Theorem 1]{BS2016} also splits into cases depending on whether $d$ is divisible by a prime in $(\frac{k}{2},k)$ or not, and, in the latter case, whether $p \in (\frac{k}{2},k)$ divides a single term or two terms. Upon appealing to Lemma \ref{lem:modularity to bound exponent ell} and bounding
\begin{align*}
\textup{Rad}_2(a_{i_p}a_j|j-i_p|),\textup{Rad}_2(a_{i_p}a_{i_p+p}a_{i_p+r}a_{i_p+s}rs) \leq \prod_{\substack{2 < q \leq k \\ q \neq p}} q,
\end{align*}
one obtains
\begin{align*}
\log \ell &\leq (1+o(1)) \frac{1}{2}\log k \cdot \frac{16}{3}\prod_{\substack{2 < q \leq k \\ q \neq p}} q = (1+o(1)) \frac{8}{3}\cdot \frac{\log k}{2p} \cdot k\# \\ 
&\leq (1+o(1)) \frac{8}{3}\cdot \frac{\log k}{k} \cdot k\#.
\end{align*}
Bennett and Siksek obtain their simpler bound $\ell < \exp(3^k)$ by using work of Schoenfeld \cite[p. 360]{Sch1976} to deduce
\begin{align*}
\prod_{\substack{2 < q \leq k \\ q \neq p}} q \leq \frac{1}{k} \prod_{q \leq k} q \leq \frac{1}{k} \exp(1.000081k),
\end{align*}
and then simplifying with the fact that one can take $k > 34$.
\end{remark}

\section{Proof of Proposition \ref{prop:some term divisible by many primes < k}}\label{sec:sieve methods}

We first collect some results we need for the proof of Proposition \ref{prop:some term divisible by many primes < k}.

\subsection{Ancillary results}

Let $\mathcal{Q}$ be a finite set of primes, and let $a_q \pmod{q}$ be a residue class for each $q \in \mathcal{Q}$. We write
\begin{align*}
P(\mathcal{Q}) = \prod_{q \in \mathcal{Q}} q.
\end{align*}
For $n \in \mathbb{N}$, define
\begin{align*}
F(n) = F(n,\mathcal{Q}) \coloneqq \mathbf{1}(n \not \equiv a_q \pmod{q} \text{ for any }q \in \mathcal{Q}).
\end{align*}

We study $F(n)$ using sieve methods. For simplicity, we work with basic Brun-type sieves (see \cite[Chapter 6]{FI2010}), but one could obtain better results by using more advanced sieves, such as the combinatorial beta-sieve (see \cite[Chapter 11]{FI2010}).

If $e$ is a squarefree integer composed entirely of primes in $\mathcal{Q}$, we let $\gamma_e \pmod{e}$ denote the congruence class arising via the Chinese remainder theorem such that $\gamma_e \equiv a_q \pmod{q}$ for every $q \mid \mathcal{Q}$. For $n,m \in \mathbb{N}$, we define
\begin{align*}
F^-(n) &\coloneqq \sum_{\substack{e \mid P(\mathcal{Q}) \\ n \equiv \gamma_e (e) \\ \omega(e) \leq 2m+1}} \mu(e), \\
F^+(n) &\coloneqq\sum_{\substack{e \mid P(\mathcal{Q}) \\ n \equiv \gamma_e (e) \\ \omega(e) \leq 2m}} \mu(e).
\end{align*}
We note that $F^\pm(n)$ depends on $\mathcal{Q}$, the residue classes $a_q$, and the integer $m$, but we omit this from the notation. We shall use $F^\pm (n)$ on integers of size $\ll k$, and set 
\begin{align*}
  m = \lfloor 10 \log \log k \rfloor.
\end{align*}

\begin{lemma}\label{lem:brun sieve inequality}
For any $n \in \mathbb{N}$ we have $F^-(n) \leq F(n) \leq F^+(n)$.
\end{lemma}
\begin{proof}
This is a Brun-sieve result of standard type, but we give a brief proof for completeness.

If $F(n) = 1$, then $n \not \equiv \gamma_e \pmod{e}$ for any $e \mid P(\mathcal{Q})$ with $e>1$, so $F^-(n)=F(n)=F^+(n)=1$.

Assume $F(n) = 0$, and define $\mathcal{Q}' = \mathcal{Q}'(n) = \{q \in \mathcal{Q} : n \equiv a_q \pmod{q}\}$, so that $|\mathcal{Q}'| \geq 1$. For any positive integer $t$, we then have
\begin{align*}
\sum_{\substack{e \mid P(\mathcal{Q}) \\ n \equiv \gamma_e (e) \\ \omega(e) \leq t}} \mu(e) &= \sum_{\substack{e \mid P(\mathcal{Q'}) \\ \omega(e) \leq t}} \mu(e) = \sum_{j=0}^t (-1)^j {{|\mathcal{Q}'|} \choose j} = (-1)^t {{|\mathcal{Q}'|-1}\choose t},
\end{align*}
the last equality following from the binomial theorem and the identity ${n \choose k} + {n\choose {k+1}} =  {{n+1}\choose {k+1}}$ (compare \cite[(6.6)]{FI2010}). If $t$ is even then the sum is $\geq 0$, and if $t$ is odd the sum is $\leq 0$. (These are versions of the so-called ``Bonferroni inequalities,'' see, e.g., \cite{Gal1977}).
\end{proof}

We need a lower bound for terms of the form $F(n_1)F(n_2)$. This is accomplished with a ``vector sieve'' inequality (see \cite[p. 94]{BF1994}).

\begin{lemma}\label{lem:vector sieve}
For any positive integers $n_1,n_2$ we have
\begin{align*}
F(n_1)F(n_2) &\geq F^+(n_1)F^-(n_2) + F^-(n_1)F^+(n_2) - F^+(n_1)F^+(n_2).
\end{align*}
\end{lemma}
\begin{proof}
By Lemma \ref{lem:brun sieve inequality}, we have $(F^+(n_1)-F(n_1))(F^+(n_2)-F(n_2)) \geq 0$. Rearranging gives
\begin{align*}
F(n_1)F(n_2) &\geq F^+(n_1)F(n_2) + F(n_1)F^+(n_2) - F^+(n_1)F^+(n_2).
\end{align*}
Since $F^+$ is nonnegative, we have $F^+(n_1)F(n_2) \geq F^+(n_1)F^-(n_2)$, and we argue similarly for $F(n_1)F^+(n_2)$.
\end{proof}

We must set up some notation in order to state the next result. We will have $0 \leq i_0 \leq k-1$ such that $n+i_0d^\ell$ is divisible by many primes $<k$. Let $J_1 = (\frac{k}{5},\frac{k}{4})$. If $i_0 \leq \frac{k}{2}$ and $p \in J_1$, then $i_0 + p,i_0+2p < k$. If $i_0 > \frac{k}{2}$, then $i_0-p,i_0-2p > 0$. We set $\varepsilon_1 = 1$ if $i_0 \leq \frac{k}{2}$, and $\varepsilon_1 = -1$ if $i_0 > \frac{k}{2}$. 

We let $\mathcal{Q}_1$ be the set of all odd primes $<k$ that do not divide $n+i_0d^\ell$, and for each $q \in \mathcal{Q}_1$ we have a corresponding residue class $a_q \pmod{q}$, where $q \mid (n+id^\ell)$ if and only if $i \equiv a_q \pmod{q}$. Note that if $q \in \mathcal{Q}_1$, then $i_0 \not \equiv a_q \pmod{q}$.

\begin{lemma}\label{lem:many primes when 3 divides n+id ell}
Let $0 \leq i_0 \leq k-1$ such that $n+i_0d^\ell$ is divisible by $>\pi(k)-100k^{1/(100\log \log k)}$ primes $<k$. Assume $3 \mid (n+i_0d^\ell)$. If $k$ is sufficiently large, then
\begin{align*}
\sum_{p \in J_1} F(i_0 + \varepsilon_1 p,\mathcal{Q}_1)F(i_0 + 2\varepsilon_1p,\mathcal{Q}_1) \gg \frac{k}{(\log k)^3}.
\end{align*}
\end{lemma}

\begin{remark}
We note that the ``sufficiently large'' in Lemma \ref{lem:many primes when 3 divides n+id ell} is ineffective, due to our reliance on the Bombieri--Vinogradov theorem. One could obtain an effective result by a more careful treatment of exceptional Dirichlet characters, but we have opted not to do so for the sake of simplicity.
\end{remark}

\begin{proof}[Proof of Lemma \ref{lem:many primes when 3 divides n+id ell}]
Note that $|\mathcal{Q}_1| \leq 100 k^{1/(100\log \log k)}$. We treat ``small'' and ``large'' elements of $\mathcal{Q}_1$ differently. Let $T_1 = k^{1/(99\log \log k)}$, and write
\begin{align*}
\mathcal{U}_1 &= \{q < k : q \in \mathcal{Q}_1, q \leq T_1\}, \\
\mathcal{V}_1 &= \{q < k : q \in \mathcal{Q}_1, q > T_1\}.
\end{align*}
By trivial estimation, the number of $p \in J_1$ such that $i_0 + \varepsilon_1p \equiv a_q \pmod{q}$ or $i_0 + 2\varepsilon_1p \equiv a_q \pmod{q}$ for some $q \in \mathcal{V}_1$ is trivially
\begin{align*}
\ll k \frac{|\mathcal{V}_1|}{T_1} \ll \frac{k}{(\log k)^{10}}.
\end{align*}

We aim to apply sieve methods to
\begin{align*}
\sum_{p \in J_1} F(i_0 + \varepsilon_1 p,\mathcal{U}_1)F(i_0 + 2\varepsilon_1p,\mathcal{U}_1) &= \sum_{p \in J_1} F(i_0 + \varepsilon_1 p)F(i_0 + 2\varepsilon_1p),
\end{align*}
where we suppress the notational dependence on $\mathcal{U}_1$. By Lemma \ref{lem:vector sieve}, we have
\begin{align*}
\sum_{p \in J_1} F(i_0 + \varepsilon_1 p)F(i_0 + 2\varepsilon_1p) &\geq \frac{1}{\log k}\sum_{p \in J_1}(\log p) F^+(i_0 + \varepsilon_1 p)F^-(i_0 + 2\varepsilon_1p) \\
&+ \frac{1}{\log k}\sum_{p \in J_1} (\log p)F^+(i_0 + 2\varepsilon_1p)( F^-(i_0 + \varepsilon_1 p)- F^+(i_0 + \varepsilon_1 p)) \\
&= \mathcal{M} + \mathcal{E},
\end{align*}
say.

We first investigate $\mathcal{E}$. We observe that
\begin{align*}
F^-(i_0 + \varepsilon_1 p)- F^+(i_0 + \varepsilon_1 p) &= \sum_{\substack{e \mid P(\mathcal{U}_1) \\ i_0 + \varepsilon_1 p \equiv \gamma_e (e) \\ \omega(e) = 2m+1}} \mu(e).
\end{align*}
Using our expression for $F^+$ and the triangle inequality, we have
\begin{align*}
|\mathcal{E}| &\ll \frac{1}{\log k} \sum_{\substack{e \mid P(\mathcal{U}_1) \\ \omega(e) = 2m+1}} \mu^2(e) \sum_{\substack{f \mid P(\mathcal{U}_1) \\ \omega(f) \leq 2m}} \mu^2(f) \sum_{\substack{p \in J_1 \\ i_0 + \varepsilon_1 p \equiv \gamma_e (e) \\ i_0 + 2\varepsilon_1 p \equiv \gamma_f (f)}} 1.
\end{align*}

We have
\begin{align*}
f \leq T_1^{2m} \leq (k^{1/(99\log \log k)})^{20 \log \log k} \leq k^{9/40},
\end{align*}
and similarly $e \leq k^{1/4}$, say. We claim these inequalities imply $e$ and $f$ are coprime. If there is some prime $q$ that divides both $e$ and $f$, then $i_0 + \varepsilon_1 p \equiv a_q \equiv i_0 + 2\varepsilon_1 p \pmod{q}$, and therefore $q \mid p$. But $q \leq k^{1/4}$ and $p$ is a prime $\asymp k$, so this is a contradiction.

Since $e$ and $f$ are coprime, we may combine the congruence conditions on $p$ into a single congruence condition modulo $ef$. We then drop the condition that $p$ is prime and estimate trivially to get
\begin{align*}
    |\mathcal{E}| &\ll \frac{k}{\log k}\sum_{\substack{e \mid P(\mathcal{U}_1) \\ \omega(e) = 2m+1}} \frac{\mu^2(e)}{e} \sum_{\substack{f \mid P(\mathcal{U}_1) \\ \omega(f) \leq 2m}} \frac{\mu^2(f)}{f} +k^{1/2}\ll k\sum_{\substack{e \mid P(\mathcal{U}_1) \\ \omega(e) = 2m+1}} \frac{\mu^2(e)}{e} + k^{1/2}.
\end{align*}
By a trick of Erd\H{o}s (see \cite[p. 51]{Erd1935}), we have
\begin{align*}
\sum_{\substack{e \mid P(\mathcal{U}_1) \\ \omega(e) = 2m+1}} \frac{\mu^2(e)}{e} &\leq \frac{1}{(2m+1)!}\left(\sum_{p \leq T_1} \frac{1}{p} \right)^{2m+1} \leq \left(\frac{3 \log\log k}{2m+1} \right)^{2m+1} \ll (\log k)^{-20},
\end{align*}
say, recalling that $m = \lfloor 10 \log \log k\rfloor$. We deduce that $|\mathcal{E}| \ll k(\log k)^{-20}$.

We turn now to estimating
\begin{align*}
\mathcal{M} &= \frac{1}{\log k} \sum_{p \in J_1} (\log p) F^+(i_0 + \varepsilon_1 p)F^-(i_0 + 2\varepsilon_1p) \\
&= \frac{1}{\log k}\sum_{\substack{e \mid P(\mathcal{U}_1) \\ \omega(e) \leq 2m}} \mu(e) \sum_{\substack{f \mid P(\mathcal{U}_1) \\ \omega(f) \leq 2m+1}} \mu(f) \sum_{\substack{p \in J_1 \\ i_0 + \varepsilon_1 p \equiv \gamma_e (e) \\ i_0 + 2\varepsilon_1 p \equiv \gamma_f (f)}} \log p.
\end{align*}
As in the estimation of $\mathcal{E}$ above, we have $e \leq k^{9/40}, f \leq k^{1/4}$, and $e,f$ are coprime. Since $a_q \not \equiv i_0 \pmod{q}$ for any $q \in \mathcal{Q}_1$, we see that $p$ lies in primitive residue classes modulo $e$ and $f$. By trivial estimation, we can replace the sum over primes weighted by $\log p$ by a sum over integers weighted by the von Mangoldt function. It follows that
\begin{align*}
    \mathcal{M} = \frac{1}{\log k}\sum_{\substack{e \mid P(\mathcal{U}_1) \\ \omega(e) \leq 2m}} \mu(e) \sum_{\substack{f \mid P(\mathcal{U}_1) \\ \omega(f) \leq 2m+1 \\ \gcd(f,e)=1}} \mu(f) \sum_{\substack{n \in J_1 \\ n \equiv \gamma_{e,f} (ef)}} \Lambda(n) + O(k^{99/100}),
\end{align*}
where $\gamma_{e,f}$ is some primitive residue class modulo $ef$ arising from the Chinese remainder theorem. We approximate the sum over primes to obtain $\mathcal{M} = \mathcal{M}_0 + \mathcal{M}_1$, where
\begin{align*}
    \mathcal{M}_0 &= \frac{1}{\log k}\sum_{\substack{e \mid P(\mathcal{U}_1) \\ \omega(e) \leq 2m}} \frac{\mu(e)}{\varphi(e)} \sum_{\substack{f \mid P(\mathcal{U}_1) \\ \omega(f) \leq 2m+1 \\ \gcd(f,e)=1}} \frac{\mu(f)}{\varphi(f)} \sum_{\substack{n \in J_1 \\ \gcd(n,ef)=1}} \Lambda(n), \\
    \mathcal{M}_1 &=  \frac{1}{\log k}\sum_{\substack{e \mid P(\mathcal{U}_1) \\ \omega(e) \leq 2m}} \mu(e) \sum_{\substack{f \mid P(\mathcal{U}_1) \\ \omega(f) \leq 2m+1 \\ \gcd(f,e)=1}} \mu(f)\left(\sum_{\substack{n \in J_1 \\ n \equiv \gamma_{e,f} (ef)}} \Lambda(n) - \frac{1}{\varphi(ef)}\sum_{\substack{n \in J_1 \\ \gcd(n,ef)=1}} \Lambda(n) \right).
\end{align*}

The sum $\mathcal{M}_1$ contributes only to the error term. We apply the triangle inequality and take the worst error in approximating primes in arithmetic progression. Changing variables $w = ef$ then yields
\begin{align*}
    |\mathcal{M}_1| &\leq \sum_{w \leq k^{19/40}} \tau(w) \cdot \max_{\gcd(a,w)=1} \left|\sum_{\substack{n \in J_1 \\ n \equiv a(w)}} \Lambda(n) - \frac{1}{\varphi(w)}\sum_{\substack{n \in J_1 \\ \gcd(n,w)=1}} \Lambda(n) \right|,
\end{align*}
where $\tau(w)$ denotes the number of divisors of $w$.

The error induced by removing the condition $\gcd(n,w)=1$ is $O(k^{1/2})$, say. We then apply Cauchy--Schwarz and the trivial bound 
\begin{align*}
    \left|\sum_{\substack{n \in J_1 \\ n \equiv a(w)}} \Lambda(n) - \frac{1}{\varphi(w)}\sum_{\substack{n \in J_1 \\ \gcd(n,w)=1}} \Lambda(n) \right| \ll \frac{k\log k}{w}
\end{align*}
to remove the $\tau(w)$ factor and use the Bombieri--Vinogradov theorem \cite[Theorem 17.1]{IK2004} to obtain
\begin{align*}
\mathcal{M}_1 \ll \frac{k}{(\log k)^4}.
\end{align*}

The $\mathcal{M}_0$ term contributes the main term. We may remove the condition $\gcd(n,ef)=1$ at the cost of a negligible error. We now examine the double sum over $e$ and $f$. By arguments similar to those we used in estimating $\mathcal{E}$ above, we may remove the conditions $\omega(e) \leq 2m, \omega(f) \leq 2m+1$ at the cost of an error in $\mathcal{M}_0$ of size $O(k(\log k)^{-4})$. An Euler product computation then reveals
\begin{align*}
\mathcal{M}_0 = \frac{1}{\log k} \sum_{n \in J_1} \Lambda(n) \cdot \prod_{q \in \mathcal{U}_1} \left(1 - \frac{1}{q-1} \right)\left(1 - \frac{1}{q-2} \right) + O \left(\frac{k}{(\log k)^4} \right).
\end{align*}
The product over $q$ is trivially
\begin{align*}
\geq \prod_{3 < q \leq k} \left(1 - \frac{1}{q}\right)^2 \cdot \prod_{q > 3}\frac{(q-3)q^2}{(q-1)^3},
\end{align*}
where we have used the fact that $2,3 \not \in \mathcal{U}_1$. By Mertens' theorem, this is
\begin{align*}
\gg (\log k)^{-2},
\end{align*}
and it follows by the prime number theorem that
\begin{align*}
\mathcal{M}_0 &\gg \frac{k}{(\log k)^3}.
\end{align*}
Combining this lower bound with the various error estimates yields the result.
\end{proof}

We cannot use Lemma \ref{lem:many primes when 3 divides n+id ell} when $3 \nmid (n+i_0d^\ell)$. To see this, assume $3 \nmid (n+i_0d^\ell)$. Recall there is some $i_3$ such that $3 \mid (n+id^\ell)$ if and only if $i \equiv i_3 \pmod{3}$, so $i_0 \not \equiv i_3 \pmod{3}$. Thus, it is not possible to have many primes $p$ such that  $i_0 + \varepsilon_1 p \not \equiv i_3 \pmod{3}$ and $i_0 + 2\varepsilon_1 p \not \equiv i_3 \pmod{3}$, since one of the congruences will always be satisfied.

Accordingly, we need a version of Lemma \ref{lem:many primes when 3 divides n+id ell} that applies when $3 \nmid (n+i_0d^\ell)$. We give some more notation here in order to state such a result. 

We will have $0 \leq i_0 \leq k-1$ such that $n+i_0d^\ell$ is divisible by many primes $<k$. Let $J_2 = (\frac{k}{7},\frac{k}{6})$. If $i_0 \leq \frac{k}{2}$ and $p \in J_2$, then $i_0 + p,i_0+3p < k$. If $i_0 > \frac{k}{2}$, then $i_0-p,i_0-3p > 0$. We set $\varepsilon_2 = 1$ if $i_0 \leq \frac{k}{2}$, and $\varepsilon_2 = -1$ if $i_0 > \frac{k}{2}$. 

We let $\mathcal{Q}_2$ be the set of all odd primes $<k$ that do not divide $n+i_0d^\ell$, and for each $q \in \mathcal{Q}_2$ we have a corresponding residue class $a_q \pmod{q}$, where $q \mid (n+id^\ell)$ if and only if $i \equiv a_q \pmod{q}$. Note that if $q \in \mathcal{Q}_2$, then $i_0 \not \equiv a_q \pmod{q}$.

\begin{lemma}\label{lem:many primes when 3 does not divide n+id ell}
Let $0 \leq i_0 \leq k-1$ such that $n+i_0d^\ell$ is divisible by $>\pi(k)-100k^{1/(100\log \log k)}$ primes $<k$. Assume $3 \nmid (n+i_0d^\ell)$. If $k$ is sufficiently large, then
\begin{align*}
\sum_{p \in J_2} F(i_0 + \varepsilon_2 p,\mathcal{Q}_2)F(i_0 + 3\varepsilon_2p,\mathcal{Q}_2) \gg \frac{k}{(\log k)^3}.
\end{align*}
\end{lemma}
\begin{proof}
The proof is similar to that of Lemma \ref{lem:many primes when 3 divides n+id ell}, but we must make some small adjustments to account for the prime $3$.

Let $i_3$ be such that $3 \mid(n+id^\ell)$ if and only if $i \equiv i_3 \pmod{3}$. By assumption, we have $i_0 \not \equiv i_3 \pmod{3}$. We then want to ensure we can find many primes $p \in J_2$ such that
\begin{align*}
i_0 + \varepsilon_2 p \not \equiv &i_3 \pmod{3}, \\
i_0 + 3\varepsilon_2 p \not \equiv &i_3 \pmod{3}.
\end{align*}
The second condition always holds, since $i_0 + 3\varepsilon_2 p \equiv i_0 \not \equiv i_3 \pmod{3}$. We can ensure $i_0 + \varepsilon_2 p \not \equiv i_3 \pmod{3}$ holds by restricting $p$ to lie in a suitable residue class $a \pmod{3}$ (where $a$ depends on $i_0,i_3,\varepsilon_2$).

The rest of the argument is essentially the same as that in the proof of Lemma \ref{lem:many primes when 3 divides n+id ell}, but we work with the set $\mathcal{Q}_2' = \{q \nmid (n+i_0d^\ell): 3 < q < k\}$. The main term has the shape
\begin{align*}
    \sim \frac{1}{\varphi(e)\log k} \sum_{n \in J_2} \Lambda(n) \cdot \prod_{\substack{ q \in \mathcal{U}_2 \\ q > 3}} \left(1 - \frac{1}{q-1} \right)\left(1 - \frac{1}{q-2} \right) &\gg \frac{k}{(\log k)^3}. \qedhere
\end{align*}
\end{proof}

We need a result on solutions to certain generalized Fermat equations. Recall that a prime $q$ is a Fermat prime or a Mersenne prime if, for some positive integer $n$, we have $q=2^n + 1$ or $q=2^n-1$, respectively.

\begin{lemma}\label{lem:Kraus fermat lemma}
Let $\ell \geq 7$ be a prime. Let $q$ be an odd prime that is neither a Mersenne prime nor a Fermat prime. Let $a,b,c,u,v,w$ be nonzero integers with $au^\ell + bv^\ell + cw^\ell=0$ and $\textup{Rad}(abc) \mid 2q$. Assume $a,b,c$ are $\ell$-power-free, and that $u,v,w$ are pairwise coprime. Then
\begin{align*}
\log \ell \leq 2(q-1)\log(\sqrt{8(q+1)}+1),
\end{align*}
or $u,v,w = \pm 1$.
\end{lemma}
\begin{proof}
We may assume without loss of generality that $a,b,c$ are pairwise coprime. We proceed by checking a few cases.

If $\text{Rad}(abc)=1$, then $au^\ell + bv^\ell + cw^\ell=0$ is of the form $X^\ell+Y^\ell+Z^\ell=0$, and this famously
has no solutions \cite{Wiles1995}.

If $\text{Rad}(abc)=2$, then $au^\ell + bv^\ell + cw^\ell=0$ is of the form $X^\ell+Y^\ell+2^rZ^\ell=0$, for some integer $1 \leq r < \ell$. This has no solutions if $r \geq 2$ \cite{Rib1997}, and if $r=1$ then $u,v,w = \pm 1$ \cite{DM1997}.

If $\text{Rad}(abc)=q$, then $au^\ell + bv^\ell + cw^\ell=0$ is of the form $X^\ell + Y^\ell + q^r Z^\ell = 0$, with $1 \leq r < \ell$. By \cite[Corollaire 2]{Kra1997}, we then have
\begin{align*}
\log \ell \leq \frac{q+11}{6} \log \left(\sqrt{\frac{q+1}{2}}+1 \right).
\end{align*}

Lastly, if $\text{Rad}(abc)=2q$, then $au^\ell + bv^\ell + cw^\ell=0$ is of the form $X^\ell + Y^\ell + 2^r q^s Z^\ell=0$ or $X^\ell + 2^r Y^\ell + q^sZ^\ell=0$, for $1 \leq r,s < \ell$. It then follows from methods of Kraus (see \cite[Corollaire 1]{Kra1997}, \cite[p. 48]{DS2021}) that
\begin{align*}
\log \ell &\leq 2(q-1)\log(\sqrt{8(q+1)}+1). \qedhere
\end{align*}
\end{proof}

\subsection{The proof of Proposition \ref{prop:some term divisible by many primes < k}}

We are now in a position to prove Proposition \ref{prop:some term divisible by many primes < k}.

\begin{proof}[Proof of Proposition \ref{prop:some term divisible by many primes < k}]
Assume $n+i_0d^\ell$ is divisible by $>\pi(k)-100k^{1/(100\log \log k)}$ primes $<k$. We note that we may assume $d>1$, by work of Erd\H{o}s and Selfridge \cite{ES1975}. We write $\mathcal{Q}$ for the set of odd primes $<k$ such that $q \nmid (n+i_0d^\ell)$. We break into cases, depending on whether or not $3 \mid (n+i_0d^\ell)$.

Assume $3 \mid (n+i_0d^\ell)$. By Lemma \ref{lem:many primes when 3 divides n+id ell}, there exists $\varepsilon_1 =\pm 1$ and $\gg k(\log k)^{-3}$ primes $p \asymp k$ such that both $n+(i_0+\varepsilon_1 p)d^\ell$ and $n+(i_0+2\varepsilon_1 p)d^\ell$ are not divisible by any $q \in \mathcal{Q}$. Since the number of Mersenne or Fermat primes $<k$ is $\ll \log k$, we may choose some such odd prime $p$ that is neither a Mersenne prime nor a Fermat prime. 

Therefore, if $q<k$ is an odd prime that divides $n+(i_0+\varepsilon_1 p)d^\ell$ or $n+(i_0+2\varepsilon_1 p)d^\ell$, we must have that $q \mid (n+i_0d^\ell)$. It then follows that any such $q$ is equal to $p$, which implies $\text{Rad}(a_{i_0+\varepsilon_1p}),\text{Rad}(a_{i_0+2\varepsilon_1p}) \mid 2p$. By subtracting, we obtain
\begin{align*}
a_{i_0+2\varepsilon_1p}x_{i_0+2\varepsilon_1p}^\ell - a_{i_0+\varepsilon_1p}x_{i_0+\varepsilon_1p}^\ell = n+(i_0+2\varepsilon_1p)d^\ell - (n+(i_0+\varepsilon_1p)d^\ell) = \varepsilon_1p d^\ell,
\end{align*} 
which is a generalized Fermat equation to which we may apply Lemma \ref{lem:Kraus fermat lemma}. 

In order to apply the lemma, we must verify that the $x_i$ and $d$ are pairwise coprime. The fact that $\gcd(x_i,d)=1$ follows from Lemma \ref{lem:factorization lemma} and the fact that $\gcd(n,d)=1$. If we assume for contradiction that some prime $q$ divides both of the $x_i$, then we have $q^\ell \mid pd^\ell$. We have $q \nmid d$, by the above reasoning, so $q^\ell \mid p$. However, $\ell > k$ and $p < k$, so this is a contradiction. Since $d > 1$, we conclude by Lemma \ref{lem:Kraus fermat lemma} that
\begin{align*}
\log \ell \leq 2(p-1)\log(\sqrt{8(p+1)}+1) \ll k \log k.
\end{align*}

We argue similarly if $3 \nmid (n+i_0d^\ell)$, only we use Lemma \ref{lem:many primes when 3 does not divide n+id ell} instead of Lemma \ref{lem:many primes when 3 divides n+id ell}. We thus get some $\varepsilon_2 = \pm 1$ and prime $p \asymp k$ that is not a Mersenne or Fermat prime such that neither $n+(i_0+\varepsilon_2p)d^\ell$ nor $n+(i_0+3\varepsilon_2p)d^\ell$ are divisible by any odd primes in $\mathcal{Q}$ (note that $3 \in \mathcal{Q}$). By arguing as in the previous case, we have
\begin{align*}
a_{i_0+3\varepsilon_2p}x_{i_0+3\varepsilon_2p}^\ell - a_{i_0+\varepsilon_2p}x_{i_0+\varepsilon_2p}^\ell = 2\varepsilon_2 p d^\ell,
\end{align*}
where $\text{Rad}(a_{i_0+\varepsilon_2p}),\text{Rad}(a_{i_0+3\varepsilon_2p}) \mid 2p$. We then apply Lemma \ref{lem:Kraus fermat lemma} as before.
\end{proof}

\section*{Acknowledgments}
The first titled author's research is partially supported by the National Science Foundation (DUE-2500154).\\

The second author is partially supported by the National Science Foundation (DMS-2418328) and the Simons Foundation (MPS-TSM-00007959). \\

The authors used ChatGPT for literature searches, LaTeX assistance, proofreading, and checking previous drafts of the paper for possible mathematical errors. All the substantial mathematical ideas in the proof originated with the authors.

\bibliographystyle{plain}
\bibliography{refs}

\end{document}